\documentclass{amsart}
\usepackage[english]{babel}
\usepackage{amsfonts, amsmath, amsthm, amssymb,amscd,indentfirst}

\usepackage{xcolor}
\usepackage{MnSymbol}
\usepackage{esint}
\usepackage{amsmath,amssymb,latexsym,indentfirst}
\usepackage{times}
\usepackage{palatino}

\newtheorem{theorem}{Theorem}[section]

\newtheorem{proposition}{Proposition}[section]

\newtheorem{definition}{Definition}[section]

\newtheorem{corollary}{Corollary}[section]

\newtheorem{remark}{Remark}[section]

\newenvironment{dedication}
{
	\itshape             
	\raggedleft          
}
{\par 
}

\begin{document}

\title[Gauss-Bonnet formula]{A probabilistic proof of the Gauss-Bonnet formula for manifolds with boundary}


\author{Levi Lopes de Lima}

\thanks{Universidade Federal do Cear\'a,
	Departamento de   Matem\'atica, Campus do Pici, R. Humberto Monte, s/n, 60455-760,
	Fortaleza/CE, Brazil.}

\email{levi@mat.ufc.br}

\thanks{Partially suported by  CNPq grant 311258/2014-0 and FUNCAP/CNPq/PRONEX Grant 00068.01.00/15.}

\maketitle

\begin{dedication}
	{\centerline {To the memory of Elon Lages Lima.}}
\end{dedication}

\begin{abstract}
In this short note we outline a simple probabilistic proof of the Gauss-Bonnet formula for compact Riemannian manifolds with boundary, which adapts to this setting an argument due to Hsu \cite{Hs1,Hs2} in the closed case. The new technical ingredient is the Feynman-Kac formula for differential forms satisfying absolute boundary conditions proved in \cite{dL}. Combined with the so-called supersymmetric aproach to index theory, this leads to a path integral representation of the Euler characteristic of the manifold in terms of normally reflected Brownian motion  
whose short time asymptotics clarifies the role played by the shape operator in determining the boundary contribution to the formula. As a consequence 
we obtain the expected {\em local} Gauss-Bonnet formula which upon integration yields the desired global result.      
\end{abstract}

\section{Introduction}\label{intro}

Let $(X,g)$ be a closed, oriented Riemannian manifold of even dimension $n\geq 2$. The Gauss-Bonnet formula computes the Euler characteristic of $X$ as 
\begin{equation}
\label{gbclosed}\chi(X)=\int_{X}{\mathcal P}(R)\,dX,
\end{equation} 
where ${\mathcal P}(R)$ is a universal polynomial expression in the curvature tensor $R$ of $g$.  In case $X$ carries a smooth boundary $Z$ one has
\begin{equation}
\label{gbvoundary}
\chi(X)=\int_{X}{\mathcal P}(R)\,dX+\int_{Z}{\mathcal Q(R,A)}\,dZ,
\end{equation}
where $\mathcal Q(R,A)$ is another universal expression which also depends on the shape operator $A$ of $Z$. Taken together, these results express a topological invariant of a compact manifold by integrating locally defined, metric dependent quantities and therefore are cornerstones in Riemannian Geometry.

The case $n=2$ is of course a classical accomplishment. In 1925, H. Hopf \cite{H} started the higher dimensional saga by establishing (\ref{gbclosed}) for hypersurfaces in $\mathbb R^{n+1}$. Combining this with Weyl's tube formula, Allendoerfer \cite{Al} proved (\ref{gbclosed}) in case the metric is induced by an isometric embedding $X\hookrightarrow\mathbb R^N$ for some $N$. Also, Fenchel \cite{F}  gave an independent proof of this same result. An authorative account of this stage of the theory can be found in \cite{Gr}. We remind however that Nash's embedding theorem was not available at the time, so Allendoerfer and Weil \cite{AW} tackled the general problem by first proving a version of (\ref{gbvoundary})  for  embedded Riemmanian polyhedra $P\hookrightarrow\mathbb R^N$ and then applying the result to the cells of a sufficiently refined triangulation of the given manifold, which is justified since any such cell could be  isometrically embedded by a classical result due to Cartan and Janet. Summing up over the cells and performing the necessary cancellations, proofs of (\ref{gbclosed}) and (\ref{gbvoundary}) emerge in full generality.

A major breakthrough occurred soon afterwards  when Chern \cite{Ch1,Ch2} discovered  {\em intrinsic} proofs of (\ref{gbclosed}) and (\ref{gbvoundary}), an accomplishment so famous that usually these expressions are termed {\em Gauss-Bonnet-Chern formulae}. Later on, with the advent of index theory, it was realized that the Euler characteristic is the Fredholm index of a natural elliptic operator acting on differential forms, which led to attempts to prove the formula by asymptotic methods \cite{MS}. In the closed case, this was first achieved by Patodi \cite{P} and then, in a broader context, by Gilkey \cite{G1} and Atiyah-Bott-Patodi \cite{ABP}.  Finally, a heat equation proof of (\ref{gbvoundary}) appeared in \cite{G2}. 

The refinement of the supersymmetric approach by physicists \cite{A-G,FW} eventually led to a whole new generation of heat equation proofs in index theory \cite{BGV, R, Y}. Almost simultaneously, probabilistic methods started being effectively used to access these results \cite{B}. In the context of (\ref{gbclosed})  above, a rather transparent argument along these lines appears in \cite{Hs1,Hs2}, where it is shown that 
the corresponding {\em local} Gauss-Bonnet formula is a direct consequence of the path integral representation of the heat semigroup of the Hodge Laplacian acting on differential forms derived from a suitable Feynman-Kac formula.   
The situation is much  more complicated in the presence of a boundary essentially because, as explained in \cite{dL}, 
the relevant boundary condition for differential forms is of mixed type, so that a naive application of It\^o's calculus fails to provide an appropriate Feynman-Kac formula for the corresponding Hodge Laplacian. In particular, the problem of figuring out in this context how the extrinsic geometry of the boundary, impersonated by the shape operator, contributes to the final formula in (\ref{gbvoundary}) is far from trivial. We note however that 
a rather involved solution to this problem, which relies on  an excursion theory for the reflected Brownian motion developed in \cite{IW}, has appeared in \cite{SUW}.

The purpose of this note is to point out that the simple and elegant  argument in \cite{Hs1,Hs2} can be adapted to handle (\ref{gbvoundary}) above. The new technical input is the Feyman-Kac formula for differential forms established in \cite{dL}; see Theorem \ref{feyn-kac-form} below. This new formula clarifies the role played by the shape operator in the path integral representation of the heat semigroup associated to the Hodge Laplacian acting on differential forms satisfying absolute boundary conditions and allows us to easily carry out the ``fantastic cancellations'' leading to (\ref{gbvoundary}). In particular, the argument dispenses altogether with the modified Malliavin calculus used in \cite{SUW}. Satisfyingly enough, the method also gives the Gauss-Bonnet formula for $n$  odd, which actually reduces to a purely topological statement, namely,  
\begin{equation}
\label{gb-impar}
\chi(X)=\frac{1}{2}\chi(Z).
\end{equation}

The main results here are  Theorems  \ref{asymstr} and \ref{asymstr2} below, which establish the local Gauss-Bonnet formulae in the even and odd dimensional cases, respectively. Their proofs, which  are carried out  in Section \ref{proof-gb}, rely on a path integral representation for the Euler characteristic whose proof combines the Feynman-Kac formula, which is described in Section \ref{feyn-kac}, with   the so-called supersymmetric approach to index theory,  reviewed in Section \ref {supersetup}. We also collect in   Appendix \ref{appbridge} the technical facts on the reflected Brownian holonomy needed in the argument.

\vspace{0.3cm}
\noindent{\bf Acknowledgements:} Part of the research leading to this work was done while the author was visiting the Princeton University Mathematics Department (feb\-ru\-ary/2017). He would like to thank Prof. F. Marques for the invitation.

\section{The Feynman-Kac formula and its consequences}\label{feyn-kac} 

Here we review the Feynman-Kac formula proved in \cite{dL}; see also \cite{Hs3} for an important previous contribution. 
As usual, we resort to the  so-called Eells-Elworthy-Malliavin approach to stochastic analysis on manifolds, as exposed in \cite{Hs2}.

Let $(X,g)$ be a compact Riemannian manifold of dimension $n\geq 2$, which we assume to be oriented. We denote by $\nabla$ the Levi-Civita connection acting on tensor fields over $X$ and assume that  the (non-empty) boundary $Z$ is oriented  by the inward unit normal vector field $\nu$. For simplicity of notation, in the following we identify $p$-forms to $p$-vectors in the usual manner, relying on  the tacit assumption that all vector spaces in sight are endowed with  natural inner products induced by $g$.

\begin{definition}
	\label{nothomdefab}
We say that a (not necessarily homogeneous) differential form $\omega$ on $X$ satisfies {\em absolute boundary conditions} (or it is {\em absolute}, for simplicity) if there holds
\begin{equation}\label{neweq}
\nu\righthalfcup\omega=0,\quad \nu\righthalfcup d\omega=0,
\end{equation}
along $Z$. 
\end{definition}

The elementary identity $I=\nu\righthalfcup\nu\wedge+\nu\wedge\righthalfcup\nu$ induces an orthogonal decomposition
\begin{equation}\label{decomport}
	\wedge TX|_{Z}={\rm Ran}\, \nu\righthalfcup\nu\wedge\oplus\,{\rm Ran}\,\nu\wedge\righthalfcup\nu,
\end{equation}
corresponding to the standard splitting of $\omega$ into its tangential and normal components. We denote by $\Pi_{\rm tan}$ and $\Pi_{\rm nor}$ the corresponding orthogonal projections, respectively. We then say that a form $\omega$ is {\em tangential} if $\Pi_{\rm nor}\omega=0$. 
Thus, (\ref{neweq}) says that both $\omega$ and $d\omega$ are tangential. 

For our purposes, the differential condition in (\ref{neweq}) should  somehow be rephras\-ed in terms of the shape operator of $Z$. This requires an algebraic formalism that we now recall. 
If $E$ is a finite dimensional vector space and $B\in {\rm End}(E)$, we denote by $DB$ the natural extension of $B$ to ${\rm End}(\wedge E)$ as a derivation. Notice that $DB$ preserves degrees of forms and accordingly  we set $DB_p=DB|_{\wedge^p E}$. Also, if $S=T\otimes U\in{\rm End}(E)\otimes {\rm End}(E)$ we define $DS\in{\rm End}(\wedge E)$ by $DS=-DT\circ DU$, so that $DS_p=DS|_{\wedge^pE}$. 

We apply this first with $B=A$, where $A=-\nabla\nu$ is the shape operator of $Z$, which we extend to $TX$  by declaring that $A\nu=0$. 
As another instance of the formalism in action  take $S=R$, the curvature tensor of $(X,g)$. Then  the celebrated Weitzenb\"ock decomposition on $p$-forms reads as
\begin{equation}
\label{weitz}\Box_p =\Delta_p +DR_p,
\end{equation} 
where $\Box=(d+d^*)^2=dd^*+d^*d$ is the Hodge Laplacian, with $d^*$ being the $L^2$ adjoint of the exterior differential $d$,  and  $\Delta$ is the Bochner Laplacian. 

With this terminology at hand, the following reformulation of (\ref{neweq}) is crucial in our context. 

\begin{proposition}\cite[Proposition 5.1]{dL}
	\label{abscond2}
	A differential form $\omega$ is absolute if and only if 
	\begin{equation}
	\label{absbound3}
	\Pi_{\rm nor}\omega=0,\quad \Pi_{\rm tan}(\nabla_\nu-DA)\omega=0.
	\end{equation}
\end{proposition}

Thus,  absolute boundary conditions are of mixed type in the intermediate range $1\leq p\leq n-1$, that is, they are Dirichlet in  normal directions and Robin in tangential directions.  


Let $x_t$, $t\geq 0$, be the (normally) reflected Brownian motion on $X$ starting at some $x_0$. Recall that $x_t=\pi {\widetilde x}_t$, where $\pi:P_{\rm SO}(X)\to X$ is the principal bundle of oriented orthonormal frames and ${\widetilde x}_t$ is the {\em horizontal}\footnote{In what follows we use a tilde to indicate the horizontal lift of semimartingales to $P_{\rm SO}(X)$.} reflected Brownian motion starting at $\widetilde x_0$, whose anti-development is the standard Brownian motion $b_t$ in $\mathbb R^n$.
Formally, ${\widetilde x}_t$ satisfies the stochastic differential equation 
\begin{equation}
\label{stocut}
d{\widetilde x}_t=\sum_{i=1}^nH_i({\widetilde x}_t)\circ db^i_t+\nu^\dagger({\widetilde x}_t)d\lambda_t,
\end{equation}
where $\{H_i\}_{i=1}^n$ are the fundamental horizontal vector fields, the dagger means the standard equivariant lift (scalarization) of tensor fields on $X$ to $P_{{\rm SO}}(X)$ and 
$\lambda_t$ is the boundary local time associated to $x_t$. A crucial observation here is that $\lambda_t$ is a non-decreasing process which only increases when the path hits the boundary. 

As already mentioned, a naive application of It\^o's calculus fails to deliver a Feynman-Kac formula describing the action of the heat semigroup $e^{-\frac{1}{2}t\Box}$ on absolute differential forms, the reason being that this formalism turns out to be  unable to detect the orthogonal projections appearing in (\ref{absbound3}); see \cite[Section 5]{dL} for a detailed discussion of this point.  
To remedy this, we follow \cite{Hs3} and consider $M_{\epsilon,t}\in{\rm End}(\wedge\mathbb R^n)$ satisfying 
\begin{equation}\label{eqmf1}
dM_{\epsilon,t}+M_{\epsilon,t}\left(\frac{1}{2}DR^\dagger({\widetilde x}_t)dt+DA_{\epsilon}^\dagger({\widetilde x}_t) d\lambda_t\right)=0, \quad M_{\epsilon,0}=I,
\end{equation}
where $\epsilon>0$ and
\begin{equation}\label{pertA}
DA_{\epsilon}^\dagger=DA^\dagger+\epsilon^{-1}\Pi^\dagger_{\rm nor}. 
\end{equation}
Note that 
\begin{equation}
\label{redAform}
DA_{\epsilon}^\dagger\omega=DA^\dagger\omega,
\end{equation}
in case $\omega$ is absolute.
It is known that, as $\epsilon\to 0$, $M_{\epsilon,t}$ converges in $L^2$ to an adapted, right-continuous multiplicative functional $M_t$ with left limits satisfying
\begin{equation}
\label{heateq}
M_t\Pi_{\rm nor}^\dagger(\widetilde x_t)=0, \quad \widetilde x_t\in\pi^{-1}(Z);
\end{equation} 
see \cite{Hs3,dL}. 

Now, let $\omega_0$ be an absolute form, so that \begin{equation}\label{fundsol0}
\omega_t=e^{-\frac{1}{2}t\Box}\omega_0
\end{equation}
 is the solution to the initial-boundary value problem
\begin{equation}\label{fundsol}
\frac{\partial \omega_t}{\partial t}+\frac{1}{2}\Box\omega_t=0, \quad \lim_{t\to 0}\omega_t=\omega_0,\quad \nu\righthalfcup\omega_t=0,\quad \nu\righthalfcup d\omega_t=0.  
\end{equation}
Then a simple application of It\^o's formula to the process $M_{\epsilon,t}\omega^\dagger_{T-t}({\widetilde x}^t)$, $0\leq t\leq T$,  confirms that, in the limit $\epsilon\to 0$, It\^o's calculus is compatible with the mixed nature of the boundary conditions in (\ref{absbound3}). In this way the following fundamental Feynman-Kac formula  is verified. 

\begin{theorem}
	\label{feyn-kac-form}\cite[Theorem 5.2]{dL}
	Under the conditions above, 
	\begin{equation}
	\label{feyn-kac2}
	\omega^\dagger_t({\widetilde x}_0)=\mathbb E_{{\widetilde x}_0}(M_t\omega^\dagger_0({\widetilde x}_t)).
	\end{equation}
	Equivalently, 
	\begin{equation}
	\label{feyn-kac3}
	\omega_t(x_0)=\mathbb E_{x_0}(M_tV_t\omega_0(x_t)),
	\end{equation}
	where $V_t=U^{-1}_t$ and $U_t$ is the stochastic parallel transport acting on differential forms.    
\end{theorem}

Now if 
\[
K(t;x,y)=\sum_{p=0}^nK_p(t;x,y)\in{\rm End}(\wedge T_yX,\wedge T_xX)
\] 
is the heat kernel, i.e. the fundamental solution of (\ref{fundsol}), we also have 
\begin{equation}
\label{heatkernel}
\omega_t(x)=\int_XK(t;x,y)\omega_0(y)dX_y. 
\end{equation} 
In order to compare this with the path integral representation in (\ref{feyn-kac3}), we consider $w_{s;x,y}$, $0\leq s\leq t$, the reflected Brownian bridge from $x$ to $y$ with lifetime $t$, which is discussed in detail in Appendix \ref{appbridge}. It then follows that 
\[
\label{compat}K(t;x,y)=K_0(t;x,y)\mathbb E_{t;x,y}(M_tV_t), 
\] 
where $K_0$ is the Neumann heat kernel acting on functions, $V_t$ is viewed under $\mathbb P_{t;x,y}$, the law of $w_{s;x,y}$, and $\mathbb E_{t;x,y}$ is the corresponding expectation. In particular, along the diagonal we have 
\begin{equation}
\label{compat}K(t;x,x)=K_0(t;x,x)\mathbb E_{t;x,x}(M_tV_t), 
\end{equation} 
where here $V_t$ is the  reflected Brownian holonomy along (reversed) loops pinned  at $x$.

\begin{remark}
	\label{hist}{\rm At least for $1$-forms, a path integral representation for solutions of (\ref{fundsol}) already appears in \cite[Chapter 6]{IW}, where the result is first established for a half-space and then patched together in the usual manner in order to obtain the representation in the general case. The idea of using the discontinuous multiplicative functional $M_{t}$ to address the problem, again for $1$-forms, is due to Hsu \cite{Hs3}. Besides leading to a Feynman-Kac formula in which  the curvature operators $DR$ and $DA$ play  similar roles, this represented a major simplification in comparison with the procedure in \cite{IW}. This approach has been shown to work fine for forms of any degree in \cite{dL}, thus yielding Theorem \ref{feyn-kac-form}.}
\end{remark}

\section{The supersymmetric approach to Gauss-Bonnet}\label{supersetup}

Similarly to the arguments in \cite{G2,SUW}, the initial step in our proof of (\ref{gbvoundary}) makes use of the so-called supersymmetric approach to index theory. We include here  a brief description of this formalism in our setting; the closed case is covered in detail in  \cite{CFKS,Hs2,Ro}.

Let
\[
\mathcal A^{\bullet}(X)=\bigoplus_{p=0}^n\mathcal A^p(X)
\]
be the space of smooth differential forms on $X$. 
Notice that 
\begin{equation}
\label{decomp}
\mathcal A^\bullet(X)=\mathcal A^+(X)\oplus \mathcal A^-(X), 
\end{equation}
where 
\[
\mathcal A^\pm(X)=\bigoplus_{(-1)^p=\pm 1} \mathcal A^p(X),
\] 
which introduces a $\mathbb Z_2$-grading on $\mathcal A^\bullet(X)$ associated to 
the parity operator $\varepsilon$ given by
\[
\varepsilon|_{\mathcal A^\pm(X)}=\pm{\rm Id}|_{\mathcal A^\pm(X)}.
\]

If we impose absolute boundary conditions, the Hodge Laplacian $\Box$ defines a well-posed elliptic boundary value problem \cite{S}, so we may consider the heat semigroup $e^{-\frac{1}{2}t\Box}$ appearing in 
(\ref{fundsol0})-(\ref{fundsol}). 
We have the corresponding 
spectral 
decomposition
\[
L^2\mathcal A^\bullet(X)=\mathcal H_{0,a}\bigoplus\oplus_{\mu_i> 0}\mathcal H_{\mu_i,a},
\]
where $\mathcal H_{\mu,a}=\{\omega\in \mathcal A^\bullet_a(X); \Box\omega=\mu\omega\}$
and the subscript $a$ indicates that absolute boundary conditions are required. By Hodge-de Rham theory we then have 
\begin{equation}\label{eulerexp}
\chi(X)=\dim\mathcal H^+_{0,a}-\dim\mathcal H^-_{0,a},
\end{equation}
where $\mathcal H_{\mu_i,a}^\pm=\mathcal H_{\mu_i,a}\cap\mathcal A^\pm(X)$;
see \cite{S}. 

From (\ref{heatkernel}) we know that $e^{-\frac{1}{2}t\Box}$ is trace class, so we may consider the {\em Witten index} $W_t={\rm Str}\,e^{-\frac{1}{2}t\Box}$, where ${\rm Str}={\rm trace}\circ\varepsilon$. Clearly,
\begin{equation}
\label{strexp}
W_t=\sum_{\mu_i\geq 0}e^{-\frac{1}{2}t\mu_i}\left({\rm dim}\,\mathcal H_{\mu_i,a}^+-{\rm dim}\,\mathcal H_{\mu_i,a}^-\right).
\end{equation}
Now 
define the Dirac-type operators $D=d+d^*:\mathcal A^\bullet(X)\to \mathcal A^\bullet(X)$ and 
\[
D_\pm=D|_{\mathcal A^\pm(X)}:\mathcal A^\pm(X)\to \mathcal A^\mp(X). 
\] 
With respect to the grading (\ref{decomp}),
\[
\Box=D^2=\left(
\begin{array}{cc}
\Box_+ &  0\\
0  & \Box_-
\end{array}
\right),\quad \Box_\pm=D_\mp D_\pm.
\]

\begin{proposition}
	\label{welldef}
	If $\mu_i>0$ then $D_{\pm}:\mathcal H_{\mu_i,a}^\pm\to \mathcal H_{\mu_i,a}^\mp$ are isomorphisms.
\end{proposition}

\begin{proof}
	First we need to check that these operators are well defined.  We first show that they preserve absolute boundary conditions. 
	We shall use that  $\Pi_{\rm tan}d=d\Pi_{\rm tan}$, $\Pi_{\rm nor}d^*=d^*\Pi_{\rm nor}$ and that the Hodge star operator $\star$ intertwines the projections, that is, $\Pi_{\rm nor}\star=\star\Pi_{\rm tan}$ and $\Pi_{\rm tan}\star=\star\Pi_{\rm nor}$; see \cite[Proposition 1.2.6]{S}. Clearly, it suffices to prove the property for $D$. Let us take $\omega\in \mathcal H_{\mu_i,a}$, so that  $\Pi_{\rm nor}\omega=0$, $\Pi_{\rm nor}d\omega=0$ and $\Box\omega=\mu_i\omega$.  We have  
	\begin{eqnarray*}
	\Pi_{\rm nor} D\omega & = & \Pi_{\rm nor}d\omega + \Pi_{\rm nor}d^*\omega\\
	& = & d^* \Pi_{\rm nor}\omega\\
	& = & 0, 
	\end{eqnarray*} 
and similarly,
\begin{eqnarray*}
\Pi_{\rm nor}dD\omega & = & \Pi_{\rm nor}d^2\omega + \Pi_{\rm nor}dd^*\omega\\
& = & \mu_i\Pi_{\rm nor}\omega -\Pi_{\rm nor}d^*d\omega \\
& = & - d^*\Pi_{\rm nor}   d\omega\\
& = & 0, 
\end{eqnarray*}
which shows that $D\omega$ is absolute indeed.

Now
take  $\varphi\in\mathcal H_{\mu_i,a}^+$, $\varphi\neq 0$, for some $\mu_i>0$. Thus, 
$D_-D_+\varphi=\Box_+\varphi=\mu_i\varphi$ and $D_+\varphi=0$ implies $\varphi=0$, a contradiction. Thus, $D_+\varphi\neq 0$ lies in $\mathcal A^-(X)$ and 
\[
\Box_- D_+\varphi=D_+D_-D_+\varphi=D_+(\Box_+\varphi)=\mu_i D_+\varphi,
\]
so that $D_+\varphi\in \mathcal H_{\mu_i,a}^-$. This shows that $D_+:\mathcal H_{\mu_i,a}^+\to \mathcal H_{\mu_i,a}^-$ is injective. On the other hand, if $\psi\in \mathcal H_{\mu_i,a}^-$, $\psi\neq 0$, then by the previous argument we know that $D_-\psi\in \mathcal H_{\mu_i,a}^+$ is non-zero and $D_+(\mu_i^{-1}D_-\psi)=\mu_i^{-1}\Box_-\psi=\psi$, 
showing that $D_+:\mathcal H_{\mu_i,a}^+\to \mathcal H_{\mu_i,a}^-$ is surjective. 
A similar argument shows that $D_-:\mathcal H_{\mu_i,a}^-\to \mathcal H_{\mu_i,a}^+$ is an isomorphism as well.
\end{proof}

By (\ref{eulerexp}) and (\ref{strexp}), this pairwise cancellation of eigenspaces in positive energy levels implies that $W_t$ actually does not depend on $t$. In fact,
$
W_t=\chi(X)
$. 
Combining this with the 
heat kernel representation (\ref{heatkernel}) we finally obtain the 
celebrated {\em McKean-Singer formula} for the Euler characteristic: 
\begin{equation}\label{heatchi}
\chi(X)=\int_X{\rm Str}\,K(t;x,x)\,dX_x, \quad t>0.
\end{equation}

Just to put our contribution in its proper perspective, we now say a few words about the conventional, non-probabilistic approach to Guass-Bonnet via (\ref{heatchi}). This relies on the prospect that as $t\to 0$ the integral in the right-hand side above splits as a sum of an integral over $X$ and another one over a tiny collar neighborhood of $Z$, so that  (\ref{gbvoundary}) is retrieved in the limit. Indeed, from general principles \cite{Gre} it is known that as $t\to 0$ the trace of the heat kernel $K^\pm$ of $e^{-\frac{1}{2}t\Box_\pm}$  admits the asymptotic expansion 
\[
{\rm trace}\,K^\pm(t;x,x)\sim \sum_{d\geq 0}\left(\left(A^\pm_{d}(x)+B^\pm_{d}(x)\right)\right)t^{-n/2+d/2},
\] 
where $B^\pm_{0}\equiv 0$ and for $d\geq 1$, $B^\pm_{d}$ localizes around $Z$ in the sense that it dies out exponentially as $t\to 0$ outside any given  neighborhood of the boundary. 
It follows from (\ref{heatchi}) that
\[
\chi(X)=\int_X\left(A^+_{n}(x)-A^-_n(x)\right)dX_x+\int_{Z}\left(
B^+_{n}(x)-B^-_n(x)\right)
dZ_x.
\]
Thus, the usual heat equation proof 
of the Gauss-Bonnet formula ultimately reduces to checking that the integrands above can be identified to the respective Gauss-Bonnet integrands in (\ref{gbvoundary}). These are the ``fantastic cancellations'' mentioned in \cite{MS}, a terminology justified by the fact that the coeffients $A^\pm_{d}$ and $B^\pm_{d}$ in the heat trace expansions above depend in principle on higher order derivatives of the metric whereas the Gauss-Bonnet integrand only involves derivatives up to second order. Nonetheless, a proof of (\ref{gbvoundary}) has been carried out along these lines in \cite{G2}. In fact, it is also proved there that  
\[
A^+_{d}-A^-_d=0, \quad  B^+_{d}-B^-_d=0, \qquad 0\leq d<n,
\]
which then implies that the integrand in (\ref{heatchi}) converges, as $t\to 0$, to the Gauss-Bonnet integrands. This is the celebrated {\em local} Gauss-Bonnet formula. We shall prove similar results in Theorem \ref{asymstr} and \ref{asymstr2} below, 
but here we shall take an alternate route dictated by our usage of stochastic machinery. More precisely, instead of relying on (\ref{heatchi}), our proof of the local Gauss-Bonnet formula confirms that the integrand in the path integral representation for the Euler characteristic in (\ref{witten}) below converges as $t\to 0$ to the Gauss-Bonnet integrands  in (\ref{gbvoundary}).

\section{The proof of the Gauss-Bonnet formula}\label{proof-gb}

We now outline our proof of the Gauss-Bonnet formula (\ref{gbvoundary}). 
Of course we adhere here to the probabilistic tenet and instead of using (\ref{heatchi}) as in \cite{G2} we follow \cite{Hs2} and appeal 
to (\ref{compat}) in order to obtain a path integral representation for the Euler characteristic, namely, 
\begin{equation}\label{witten}
\chi(X)=\int_XK_0(t;x,x)\,\mathbb E_{t;x,x}\, {\rm Str}\,(M_tV_t)\,dX_x, \quad t>0.
\end{equation}
Roughly speaking, the strategy consists in  
first showing that, as $t\to 0$, the quantity  $\mathbb E_{t;x,x}{\rm Str}\,M_tV_t\,dX$ admits a {\em polynomial} expansion in powers of $t^{1/2}$ (of course this becomes an expansion in powers of $t$ in the closed case, as already explained in \cite{Hs2}, so the half-integer powers arise because of the boundary). 
Taking into account the well-known fact that $K_0(t;x,x)\sim(2\pi t)^{-n/2}$ as $t\to 0$, we next check that in this expansion only the terms involving the power $t^{n/2}$ survive in the asymptotic limit and moreover that the corresponding coefficients yield the expected Gauss-Bonnet integrands. As we shall see below, these cancellations follow rather straightforwardly from the specific differential equation (\ref{eqmf1}) satisfied by the multiplicative functional $M_{\epsilon,t}$ and standard estimates on the reflected Brownian holonomy $V_t$.     
This line of thought will eventually lead to the following remarkable result, first considered in \cite{G2} in connection with (\ref{heatchi}) above; see also \cite{SUW}.
 
\begin{theorem}
	\label{asymstr}(Local Gauss-Bonnet with boundary in even dimension) If $n\geq 2$ is even then  
	\begin{eqnarray}\label{asymstrint}
	\lim_{t\to 0}\,(2\pi t)^{-n/2}\mathbb E_{t;x,x}{\rm Str}\,M_tV_tdX 
	 & = &  b_n{\rm Str}\,DR^{n/2}dX\nonumber\\
	& & \quad +\left(\sum_{2k+l=n-1}b_{n,k,l}{\rm Str}\,DR^{k}DA^l\right)dZ,
	\end{eqnarray} 
	for some universal constants $b_n$ and $b_{n,k,l}$.
\end{theorem}

We first explain how the bulk term $b_n{\rm Str}\,DR^{n/2}$
shows up in the asymptotic limit. 
Observe that $d\lambda_t=0$ on $X\backslash Z$ so that there (\ref{eqmf1})  reduces to
\[
dM_t+\frac{1}{2}M_tDR^\dagger(\widetilde x_t)dt=0, 
\]
with no dependence whatsoever on $\epsilon$.
Thus we may assume that $Z$ is empty, so the reasoning in \cite{Hs2} applies. Of course this is just a manifestation of the well-known ``principle of not feeling the boundary'', formally justified by means of Proposition \ref{notfeel1}, which guarantees that, up to a $\mathbb P_{t;x,x}$-negligible set, the Brownian loop pinned at $x\in M\backslash Z$ remains uniformly away from $Z$ as $t\to 0$. In any case, we present here  the argument for the sake of completeness. By iterating 
\[
M_t=I-\frac{1}{2}\int_0^t M_\tau DR^\dagger(\widetilde x_\tau)\,d\tau
\]
we find that 
\begin{equation}
\label{iterint}
M_t=\sum_{i\leq n/2}\frac{(-1)^i}{2^i}\mathcal M_i(t)+O(t^{n/2+1}),
\end{equation}  
where each $\mathcal M_i$ has ``curvature degree'' $i$, i.e., 
\begin{equation}
\label{expinter}
\mathcal M_i(t)=\int_0^t\int_0^{\tau_i}\cdots\int_0^{\tau_2}DR^\dagger(\widetilde x_{\tau_i})\cdots DR^\dagger(\widetilde x_{\tau_1})\,d\tau_1\cdots d\tau_i,
\end{equation}
where $0\leq \tau_{1}\leq \tau_2\leq\cdots\leq \tau_i\leq t$.
On the other hand, if $t$ is small we have $V_t=e^{Dv_t}$ for some ${\mathfrak {so}}_n$-valued process $v_t$. From Proposition \ref{pinnedest} we thus have 
\[
V_t=\sum_{j=0}^{n/2}\frac{1}{j!}Dv_t^j+r(t), 
\]
with $\mathbb E_{t;x,x}|r(t)|=O(t^{n/2+1})$.
From this and (\ref{iterint}) we obtain
\begin{equation}
\label{expprod}M_tV_t=\sum_{i,j\leq n/2}\frac{(-1)^i}{j!2^i}\mathcal M_i(t)Dv_t^j+s(t), 
\end{equation}
with $\mathbb E_{t;x,x}|s(t)|=O(t^{n/2+1})$.
Now, the standard Berezin-Patodi cancellation \cite[Corollary 7.3.3]{Hs2} applies and we get
\[
{\rm Str}\, \mathcal M_i(t)Dv_t^j=0, \quad 2i+j<n. 
\]
Since again by Proposition \ref{pinnedest} we know that, as $t\to 0$, 
\[
\mathbb E_{t;x,x}|\mathcal M_i(t)Dv_t^j|=O(t^{i+j}), 
\]
the only contributions to the asymptotic limit of the left-hand side of (\ref{asymstrint}) coming from applying $\mathbb E_{t;x,x}\,{\rm Str}$ to (\ref{expprod}) occur when $2i+j\geq n$ and $i+j\leq n/2$, that is, $j=0$ and $i=n/2$. By (\ref{expinter}) this gives the first term in the right-hand side of (\ref{asymstrint}). 

We now turn to the boundary contribution in (\ref{asymstrint}). By (\ref{eqmf1}) this time we should iterate
\[
M_{\epsilon,t}=I-\frac{1}{2}\int_0^t M_{\epsilon,\tau} DR^\dagger(\widetilde x_\tau)\,d\tau-\int_0^tM_{\epsilon,\tau}DA_\epsilon^\dagger(\widetilde x_\tau)\,d\lambda_\tau,  
\]
along Brownian loops pinned at some $x_0\in Z$.
We obtain
\begin{equation}\label{expboundcont}
M_{\epsilon,t}=\sum_{|I|\leq n/2-1/2}\frac{(-1)^{k+l}}{2^k}\mathcal N_{\epsilon,I}(t)+\phi_\epsilon(t),
\end{equation}  
where $\mathcal N_{\epsilon,I}(t)$ has ``curvature bi-degree'' $I=(k,l)$ and $|I|=k+l/2$. This means that   
\begin{eqnarray}\label{bidegree}
\mathcal N_{\epsilon,I}(t) & = & \sum\int_0^t\int_0^{\tau_{k+l}}\cdots\int_0^{\tau_2} 
DR^\dagger(\widetilde x_{\tau_{k+l}})\cdots DR^\dagger(\widetilde x_{\tau_{1+l}})DA_\epsilon^\dagger(\widetilde x_{\tau_{l}})\cdots 
DA_\epsilon^\dagger(\widetilde x_{\tau_{1}})
\times\nonumber\\
& & \quad \quad \quad\times\,
d\lambda_{\tau_{1}}\cdots d\lambda_{\tau_{l}} d\tau_{1+l} \cdots 
d\tau_{k+l},
\end{eqnarray}
where 
$0\leq \tau_1\leq\cdots\leq\tau_{l}\leq\tau_{1+l}\leq\cdots\leq\tau_{k+l}\leq t$
and 
the sum extends over expressions obtained from the displayed term by rearrangement of the  factors in its integrand. 
Moreover, the remainder $\phi_\epsilon$ collects similar terms satisfying $|I|\geq n/2$, with an $M_{\epsilon,t}$ inserted in each term with exactly $n$ curvature factors.

The localization around $Z$ mentioned above is once again a consequence of  Proposition \ref{notfeel1}, namely,
given $\rho>0$ small, up to a $\mathbb P_{t;x,x}$-negligible set the Brownian loop pinned at $x\in Z$ remains in the geodesic collar neighborhood
$
B_{\rho}(Z)
$ 
of radius $\rho$ around $Z$ as $t\to 0$. Notice that $B_\rho(Z)$ is foliated by parallel hypersurfaces (the level sets of the distance function to $Z$).
Thus, we may assume that $\rho\to 0$, so    
it suffices to evaluate (\ref{expboundcont}) on  forms of the type $V_t\omega$, where $\omega$ is absolute along the leaves of this  foliation.  
By (\ref{redAform}) this eliminates the dependence of (\ref{expboundcont}) on $\epsilon$ so in the limit $\epsilon\to 0$ we obtain    
\[
M_t=\sum_{|I|\leq n/2-1/2}\frac{(-1)^{k+l}}{2^k}\mathcal N_{I}(t)+\phi(t), 
\]
where the rightand side is obtained from (\ref{expboundcont}) after replacing $A_\epsilon^\dagger$ by $A^\dagger$.

We now recall that 
\[
\mathbb E_x\lambda_t=\int_0^t\int_{Z}K_0(s;x,y)\,dZ_y\,ds= O(t^{1/2}),
\]
so that $\mathbb E_{t;x,x}|\lambda_t|=O(t^{1/2})$ as well, which gives  $\mathbb E_{t;x,x} |\phi(t)|=O(t^{n/2})$.
It follows that 
\begin{equation}\label{restmat}
M_t V_t=\sum_{|I|,j\leq n/2-1/2}\frac{(-1)^{k+l}}{j!2^k}\mathcal N_{I}(t)Dv_t^j+\xi(t),
\end{equation}
where  
$\mathbb E_{t;x,x}|\xi(t)|=O(t^{n/2})$. Moreover, since the endomorphisms only act on absolute (and hence tangential) forms,   
the Berezin-Patodi cancellation here asserts that
\begin{equation}\label{berpat}
{\rm Str}\,\mathcal N_{I}(t)Dv_t^j=0,\quad 2|I|+j<n-1. 
\end{equation}

On the average, the remainder $\xi(t)$ in (\ref{restmat}) has exactly the same behavior as the Neumann heat kernel $K_0$ when $t\to 0$, thus yielding no asymptotic cancellation. 
We observe however that in the localization around $Z$ mentioned above there holds  $dX|_{B_\rho(Z)}=\zeta\wedge dZ$, so that for 
$\zeta(t)=\zeta(w_{t;x,x})$,
$E_{t;x,x}|\zeta(t)|$ measures the average normal displacement of the Brownian loop with respect to the boundary. 
In fact, it follows readily from (\ref{distest}) that  $\mathbb E_{t;x,x}|\zeta(t)|=O(t^{1/2})$. Thus, 
\begin{equation}\label{exprod2}
M_t V_tdX=\left(\sum_{|I|,j\leq n/2-1/2}\frac{(-1)^{k+l}}{j!2^k}\mathcal N_{I}(t)Dv_t^j\right)\zeta(t)\wedge dZ+\xi(t)\zeta(t)\wedge dZ,
\end{equation}
where $\mathbb E_{t;x,x}|\xi(t)\zeta(t)|=O(t^{n/2+1/2})$.
Combining this with (\ref{berpat}) and the fact that, again by Proposition \ref{pinnedest}, 
\[
\mathbb E_{t;x,x}|\mathcal N_I(t)Dv_t^j\zeta(t)|=O(t^{|I|+ j+1/2}),
\]
we conclude  that the only contributions to the asymptotic limit of the left-hand side of (\ref{asymstrint})
coming from applying $\mathbb E_{t;x,x}\,{\rm Str}$ to (\ref{exprod2}) occur when $2|I|+j\geq n-1$ and $|I|+j+1/2\leq n/2$, that is,  $j=0$ and $|I|=(n-1)/2$. By (\ref{redAform}) we may replace $A^\dagger_\epsilon$ by $A^\dagger$ in (\ref{bidegree}) so as to obtain the second term in the right-hand side of (\ref{asymstrint}) as $t\to 0$. This completes the proof of Theorem  \ref{asymstr}. 

An immediate consequence of Theorem \ref{asymstr}, as applied  to (\ref{witten}), is the following foundational result first proved by Allendoerfer and Weil \cite{AW}.

\begin{theorem}
	\label{allenweil} Let $(X,g)$ be a compact Riemannian manifold of even dimension $n\geq 2$ with boundary $Z$. Then 
	\[
	\chi(X)=b_n\int_X{\rm Str}\,DR^{n/2}\,dX +\sum_{2k+l=n-1}b_{n,k,l}\int_{Z}\,{\rm Str}\, DR^kDA^l
	dZ.
	\] 
\end{theorem}

One easily computes that
\[
{\rm Str}\,DR^{n/2}=c_n\sum\delta^{i_1i_2\cdots i_{n-1}i_n}_{j_1j_2\cdots j_{n-1}j_n}R_{i_1i_2}^{j_1j_2}\cdots R_{i_{n-1}i_n}^{j_{n-1}j_n},
\]
a universal multiple of the Pfaffian, with similar expansions holding
for the boundary integrands. The universal constants in the theorem can be explicitly computed either by carefully tracking the coefficients appearing in the various expansions along the proof or by simply  evaluating the formula on the manifolds $\mathbb S^l\times\mathbb D^{n-l}$, $l=0,1,\cdots,n$. In either way  we  recover the standard statement of the theorem \cite{AW, Ch2, G2, SUW}.

\begin{corollary}
	\label{coro} If $Z\subset X$ is totally geodesic then 
	\[
	\chi(X)=b_n\int_M{\rm Str}\,DR^{n/2}\,dX. 
	\] 
\end{corollary}

\begin{proof}
	Just observe that $l=n-1-2k$ is odd, so the boundary integrands always display a power of $DA$. 
\end{proof}

We now prove (\ref{gb-impar}). We first consider the local Gauss-Bonnet formula in this case.
\begin{theorem}
	\label{asymstr2}(Local Gauss-Bonnet with boundary in odd dimension) If $X$ has odd dimension $n\geq 3$ then  
	\begin{equation}\label{asymstrintodd}
	\lim_{t\to 0}\,(2\pi t)^{-n/2}\mathbb E_{t;x,x}{\rm Str}\,M_tV_tdX = d_n{\rm Str}\,DR_{Z}^{\frac{n-1}{2}}dZ,
	\end{equation} 
	for some universal constant $d_n$. 
	Here, $R_Z$ is the curvature tensor of the induced metric on $Z$. 
\end{theorem}

 	Since $n$ is odd, the argument leading to (\ref{asymstrint}) already shows that no bulk term appears in (\ref{asymstrintodd}). As for the boundary term, if we repeat the argument by taking into account that now $l=2m$ is even, we get
 \[
 \lim_{t\to 0}\,(2\pi t)^{-n/2}\mathbb E_{t;x,x}{\rm Str}\,M_tV_tdX  =  \sum_{k+m=\frac{n-1}{2}}b_{n,k,m}{\rm Str}\,DR^{k}\left(DA^{2}\right)^mdZ,
 \]
 with $DA^2=D\mathcal A$,
 where  
 \[
 \mathcal A(z_1,z_2,w_1,w_2)=\det(\langle Az_i,w_j\rangle).
 \]
 It is not hard to check that 
 \[
 b_{n,k,m}=d_n
 \left(
 \begin{array}{c}
 \frac{n-1}{2}\\
 k
 \end{array}
 \right),
 \]
 so we end up with 
 \[
 \lim_{t\to 0}\,(2\pi t)^{-n/2}\mathbb E_{t;x,x}{\rm Str}\,M_tV_tdX  =   
 d_{n}\,{\rm Str}\, D\left(R+\mathcal A\right)^{\frac{n-1}{2}}dZ.
 \]
 Now, 
 $R+\mathcal A=R_{Z}$ by Gauss equation and the proof of Theorem \ref{asymstr2} is completed.

Leading this to (\ref{witten}) and applying Theorem \ref{allenweil} to $Z$, which is a  closed even dimensional manifold,
yields $\chi(X)=e_n\chi(Z)$, for some universal constant $e_n$, and  evaluating this on a ball we conclude that $e_n=1/2$. This completes the proof of (\ref{gb-impar}) .

\appendix\section{The reflected Brownian holonomy}\label{appbridge}

In this appendix we collect the technical facts regarding the reflected Brownian holonomy  needed in previous sections. This is a direct refinement of the theory presented in \cite[Chapter 7]{Hs2}, which deals with the closed case, so we omit some details. 

As always, we consider a compact Riemannian manifold $(X,g)$ of dimension $n\geq 2$ and with smooth boundary $Z$. We denote by $x_s$, $s\geq 0$, the associated reflected Brownian motion starting at $x_0=x$, whose law is $\mathbb P_x$. Given $y\in X$ and $t>0$, the reflected Brownian bridge $w_{s;x,y}$, $0\leq s\leq t$, is the process  defined by following  $x_s$ and conditioning it to hit $y$ at time $t$.

We need to figure out the stochastic differential equation satisfied by ${w}_{s;x,y}$. 
We first observe that the law $\mathbb P_{t;x,y}$ of $w_{s;x,y}$ should be  determined, via Kolmogorov's extension theorem,  by the finite dimensional marginal densities
\[
K_0(t;x,y)^{-1}\Pi_{i=1}^r K_0(s_{i+1}-s_i;x_i,x_{i+1}), 
\] 
where $0=s_0<s_1<\cdots<s_r<s_{r+1}=t$, $x_0=x$ and $x_{r+1}=y$ and $K_0$ is the Neumann heat kernel (the transition probability density of $x_s$). 
It follows that 
\[
E_s:=\frac{K_0(t-s;{x}_{s},y)}{K_0(t;x,y)}, \quad s<t,
\]
satisfies 
\[
\frac{d\mathbb P_{t;x,y}}{d\mathbb P_x}|_{\mathcal F_s}=E_s,
\]
where $\mathcal F_*$ is the standard filtration associated to $x_s$. 

Now set 
\[
E_s^\dagger=\frac{K_0^\dagger(t-s;{\widetilde x}_{s},y)}{K_0(t;x,y)}, 
\]
where $K^\dagger_0(s;u,y)=K_0(s;\pi u,y)$, $u\in P_{\rm SO}(X)$.
Since 
$
{\partial\log E^\dagger_s}/{\partial \nu^\dagger}=0$, 
It\^o's formula gives
\[
d\log E^\dagger_s=\langle G^\dagger_s,db_s\rangle-\frac{1}{2}|G^\dagger_s|^2ds,
\] 
where $G_s=\nabla\log K_0(t-s;{x}_{s},y)$. Thus, by Girsanov's theorem, under $\mathbb P_{t;x,y}$
the process
$
b^s-\int_0^sG_\tau\,d\tau
$
is a reflected Brownian motion. Comparing with (\ref{stocut}) we conclude that ${\widetilde w}_{s;x,y}$ satisfies the stochastic differential equation
\begin{eqnarray}\label{bridgesde}
d{\widetilde w}_{s;x,y} & = & \sum_{i=1}^n H_i({\widetilde w}_{s;x,y})\circ\left(
db^i_s+H_i\log K_0^\dagger(t-s;{\widetilde w}_{s;x,y},y)ds 
\right)+\nonumber\\
& & \quad +\nu^\dagger({\widetilde w}_{s;x,y})d\lambda_s.
\end{eqnarray}
Thus, reflected Brownian bridge is just reflected Brownian motion with an added drift involving the logarithmic derivative of $K_0$.
We note however that the drift 
is singular at $s=t$. Fortunately, careful first order estimates of $K_0$ allow us to bypass this difficulty and restore the semimartingale character of ${\widetilde w}_{s;x,y}$ at $s=t$. 

\begin{proposition}
	\label{restore}
	The reflected Brownian bridge ${w}_{s;x,y}$ is a $\mathbb P_{t;x,y}$-semimartingale on the whole interval $[0,t]$. 
\end{proposition}

\begin{proof}
	We already know that $\widetilde w_{s;x,y}$ is a horizontal $\mathbb P_{t;x,y}$-semimartingale for $s<t$. This allows us to explore (\ref{bridgesde}) by applying It\^o's formula   to the function 
	\[
	J(s;u):=\log K_0^\dagger(s;u,y).
	\]
	Since $\partial  J(s;u)/\partial \nu^\dagger=0$, the extra boundary term in (\ref{bridgesde}) plays no role in the final result. Thus, the procedure in \cite[Section 5.5]{Hs2} applies with no substantial modification and we end up with the estimate
	\begin{equation}
	\label{estheatder}
	|\nabla\log K_0(s;x,y)|\leq C\left({d_X(x,y)}{s}^{-1}+s^{-1/2}\right),
	\end{equation}   
	for $(s;x,y)\in (0,1)\times X\times X$, where $d_X$ is the intrinsic distance on $X$.
	Since 
	\begin{equation}\label{distest}
	\mathbb E_{t;x,y}d_X(x, x_s)^N\leq C{s}^{N/2},
	\end{equation}
	for any integer $N>0$, 
	we  get
	\[
	\mathbb E_{t;x,y}\int_0^t|\nabla\log K_0(t-s;x_s,y)|ds<+\infty,
	\]
	which suffices to finish the proof.
\end{proof}

This result allows us to estimate the stochastic parallel transport  $U_t$ acting on differential forms along loops pinned at $x$ (Brownian holonomy). Note that if $t$ is small we have $U_t=e^{Du_t}$ for some ${\mathfrak {so}}_n$-valued process $u_t$. 

\begin{proposition}
	\label{pinnedest} For any integer $N>0$ we have 
	\[
	\mathbb E_{t;x,x}|U_t-I|^N\leq Ct^N.
	\]
	Equivalently, 
	\[
	\mathbb E_{t;x,x}|u_t|^N\leq Ct^N.
	\]
\end{proposition}

\begin{proof}
	We compare with the proof of \cite[Lemma 7.3.4]{Hs2}. The idea is to apply It\^o's formula to a suitably chosen matrix-valued test function $F$ defined in a small invariant neighborhood of $P_{\rm SO}(X)$ around $\widetilde x$. If $x\in X\backslash Z$ we choose the same $F$ as there. If $x\in Z$ we may choose $F$ so as to additionally satisfy Neumann boundary condition. In either case, the extra boundary term in (\ref{bridgesde}) contributes nothing to the final result. This allows us to express powers of $U_t-I$ as a sum of terms whose expectations can be easily estimated  by means of  (\ref{estheatder}) and (\ref{distest}).    
	\end{proof}

We also need to estimate from below the size under  $\mathbb P_{t;x,x}$ of the subset of Brownian loops pinned at $x$ which remain inside a given geodesic ball $B_\rho(x)$ of radius $\rho>0$ as $t\to 0$. 

\begin{proposition}
	\label{notfeel1}
	For any $\rho>0$ there exists a constant $C=C_{\rho}>0$ such that 
	\[
	\mathbb P_{t;x,x}\left[x_s\in B_\rho(x) \,\,\,{\rm for}\,\,\,{\rm all}\,\,\,s\leq t\right]\geq 1-e^{-C/t},
	\]
	as $t\to 0$. 
\end{proposition}

\begin{proof}
	This is just \cite[Lemma 7.7.1]{Hs2} adapted to our purposes. In light of  the discussion above, the proof can be implemented in quite the same way, so it is omitted. 
\end{proof}


\begin{thebibliography}{99999}
  
\markboth{}{}



\bibitem[Al]{Al} Allendoerfer, C. B., The Euler number of a Riemann manifold. {\em Amer. J. Math.} 62, (1940). 243-248. 

\bibitem[AW]{AW} Allendoerfer, C. B., Weil, A.
The Gauss-Bonnet theorem for Riemannian polyhedra. 
{\em Trans. Amer. Math. Soc.} 53, (1943) 101-129. 


\bibitem[A-G]{A-G} Alvarez-Gaum\'e, L., Supersymmetry and the Atiyah-Singer index theorem. {\em Comm. Math. Phys.} 90 (1983), no. 2, 161-173.


\bibitem[ABP]{ABP} Atiyah, M., Bott, R., Patodi, V. K., On the heat equation and the index theorem. {\em Invent. Math.} 19 (1973), 279-330.


\bibitem[BGV]{BGV} Berline, N., Getzler, E., Vergne, M., {\em Heat kernels and Dirac operators.}  Grundlehren Text Editions. Springer-Verlag, Berlin, 2004.

\bibitem[B]{B} Bismut, J.-M., Index theorem and the heat equation. {\em Proceedings of the International Congress of Mathematicians}, Vol. 1, 2 (Berkeley, Calif., 1986), 491–504, Amer. Math. Soc., Providence, RI, 1987.

\bibitem[Ch1]{Ch1} Chern, S.-S.,
A simple intrinsic proof of the Gauss-Bonnet formula for closed Riemannian manifolds. 
{\em Ann. of Math.} (2) 45, (1944) 747-752.

\bibitem[Ch2]{Ch2} Chern, S.-S.
On the curvatura integra in a Riemannian manifold. 
{\em Ann. of Math.} (2) 46, (1945) 674-684.


\bibitem[CFKS]{CFKS} Cycon, H. L., Froese, R. G., Kirsch, W., Simon, B. {\em Schr\"odinger operators with application to quantum mechanics and global geometry.} Texts and Monographs in Physics. Springer Study Edition. Springer-Verlag, Berlin, 1987.

\bibitem[dL]{dL} de Lima, L. L., 
A Feynman-Kac formula for differential forms on manifolds with boundary and applications,
{\em  arXiv:1512.01153}, to appear in the {Pacific Journal of Mathematics}.  

\bibitem[F]{F}
Fenchel, W.,
On total curvatures of Riemannian manifolds: I. 
{\em J. London Math. Soc.} 15, (1940). 15-22.


\bibitem[FW]{FW} Friedan, D., Windey, P., Supersymmetric derivation of the Atiyah-Singer index and the chiral anomaly. {\em Nuclear Phys. B} 235 (1984), no. 3, 395-416.

\bibitem[G1]{G1} Gilkey, P. B. {\em The index theorem and the heat equation.}  Mathematics Lecture Series, No. 4. Publish or Perish, Inc., Boston, Mass., 1974.

\bibitem[G2]{G2} Gilkey, P. B.,
The boundary integrand in the formula for the signature and Euler characteristic of a Riemannian manifold with boundary. {\em Advances in Math.} 15 (1975), 334-360.

\bibitem[Gr]{Gr} Gray, A., {\em Tubes.} Second edition. With a preface by Vicente Miquel. Progress in Mathematics, 221. Birkhäuser Verlag, Basel, 2004.

\bibitem[Gre]{Gre} Greiner, P.,
An asymptotic expansion for the heat equation. 
{\em Arch. Rational Mech. Anal.} 41 (1971), 163-218.

\bibitem[H]{H} Hopf, H.,
\"Uber die Curvatura integra geschlossener Hyperfl\"achen.  
{\em Math. Ann.} 95 (1926), no. 1, 340-367.

\bibitem[Hs1]{Hs1} Hsu, E. P.,
Stochastic local Gauss-Bonnet-Chern theorem.  
{\em J. Theoret. Probab.} 10 (1997), no. 4, 819-834.

\bibitem[Hs2]{Hs2} Hsu, E. P.,
{\em Stochastic analysis on manifolds.}  
Graduate Studies in Mathematics, 38. American Mathematical Society, Providence, RI, 2002.

\bibitem[Hs3]{Hs3} Hsu, E. P. Multiplicative functional for the heat equation on manifolds with boundary. {\em Michigan Math. J.} 50 (2002), no. 2, 351-367.

\bibitem[IW]{IW} Ikeda, N., Watanabe, S., {\em Stochastic differential equations and diffusion processes.}  North-Holland Publishing Co., Amsterdam; Kodansha, Ltd., Tokyo, 1989


\bibitem[MS]{MS}  McKean, H. P., Jr., Singer, I. M.,
Curvature and the eigenvalues of the Laplacian. 
{\em J. Differential Geometry} 1 1967 no. 1, 43-69.

\bibitem[P]{P} Patodi, V. K.,
Curvature and the eigenforms of the Laplace operator. 
{\em J. Differential Geometry} 5 (1971), 233-249.

\bibitem[R]{R} Roe, J., {\em Elliptic operators, topology and asymptotic methods.} Second edition. Pitman Research Notes in Mathematics Series, 395. Longman, Harlow, 1998.

\bibitem[Ro]{Ro}  Rosenberg, S., {\em  The Laplacian on a Riemannian manifold. An introduction to analysis on manifolds.} London Mathematical Society Student Texts, 31. Cambridge University Press, Cambridge, 1997.

\bibitem[S]{S} Schwarz, G., {\em 
Hodge decomposition - a method for solving boundary value problems.} 
Lecture Notes in Mathematics, 1607. Springer-Verlag, Berlin, 1995.


\bibitem[SUW]{SUW} Shigekawa, I., Ueki, N., Watanabe, S.,
A probabilistic proof of the Gauss-Bonnet-Chern theorem for manifolds with boundary. 
{\em Osaka J. Math.} 26 (1989), no. 4, 897-930.

\bibitem[Y]{Y} Yu, Y., {\em The index theorem and the heat equation method}. Nankai Tracts in Mathematics, World Scientific Publishing Co., Inc., River Edge, NJ, 2001.

\end{thebibliography}
\end{document}